\documentclass[12pt]{amsart}
\usepackage{amsmath,amsfonts,amssymb,amsthm,bbm,graphics,epsfig,psfrag,epstopdf,dsfont,esint,latexsym,mathtools,yhmath,geometry,paralist,subfigure,float,color,soul,hyperref}
\usepackage{graphicx}
\usepackage{pifont}
\usepackage{mathrsfs}
\usepackage{epsfig,verbatim}
\usepackage{xcolor}
\usepackage{ulem}

\numberwithin{equation}{section}

\newtheorem{theorem}{Theorem}[section]
\newtheorem{lemma}[theorem]{Lemma}

\newtheorem{proposition}[theorem]{Proposition}
\newtheorem{Definition}{Definition}[section]
\newenvironment{definition}{\begin{Definition}\rm}{\end{Definition}}
\newtheorem{Rem}{Remark}[section]

\newcommand\R{\mathbb{R}}

\newcommand\N{\mathbb{N}}

\newcommand\be{\begin{equation}}
\newcommand\ee{\end{equation}}
\newcommand\bea{\begin{eqnarray}}
\newcommand\eea{\end{eqnarray}}
\newcommand\beaa{\begin{eqnarray*}}
\newcommand\eeaa{\end{eqnarray*}}
\newcommand\bss{\begin{cases}}
\newcommand\ess{\end{cases}}

\newcommand\bR{{\Bbb R}}

\newcommand{\ep}{\varepsilon}
\newcommand{\e}{\ep}

\newcommand{\ld}{\lambda}

\newcommand\bN{\mathcal{N}}

\newcommand\upsi{{\underline{\psi}}}
\newcommand\uphi{{\underline{\phi}}}
\newcommand\ophi{{\overline{\phi}}}
\newcommand\opsi{{\overline{\psi}}}

\newcommand\mU{\mathcal{U}}
\newcommand\mL{\mathcal{L}}
\newcommand\elz{e^{-\ld_1 z}}

\begin{document}

\title[Singular predator-prey model with nonlocal dispersal]{Traveling wave solutions for a singular diffusive prey-predator model with nonlocal dispersal}

\author[J.-S. Guo]{Jong-Shenq Guo}
\address{Department of Applied Mathematics and Data Science, Tamkang University, Tamsui, New Taipei City 251301, Taiwan}
\email{jsguo@mail.tku.edu.tw}

\author[F. Hamel]{Fran{\c{c}}ois Hamel}
\address{Aix Marseille Univ, CNRS, I2M, Marseille, France}
\email{francois.hamel@univ-amu.fr}

\author[C.-C. Wu]{Chin-Chin Wu}
\address{Department of Applied Mathematics, National Chung Hsing University, Taichung 402, Taiwan}
\email{chin@email.nchu.edu.tw}



\thanks{{\sl Keywords:} Prey-predator system, nonlocal dispersal, traveling wave, minimal speed.}

\thanks{{\sl 2000 Mathematics Subject Classification:} Primary: 35K55, 35K57; Secondary: 92D25, 92D40.}

\thanks{This work is partially supported by the National Science and Technology Council of Taiwan under the
grants 114-2115-M-032-005 (JSG) and 114-2115-M-005-004 (CCW), 
and by the French Agence Nationale de la Recherche (ANR) in the framework of the ReaCh project (ANR-23-CE40-0023-02).}

\maketitle

\begin{center}
{\it To Professor Yoshihisa Morita, with admiration to an esteemed mathematician}
\end{center}

\begin{abstract}
We study a singular diffusive prey-predator system with nonlocal dispersal for which the carrying capacity of the predator is proportional to the density of prey. We show the existence of positive one-dimensional traveling waves connecting the predator-free state and the constant co-existence state. The set of admissible wave speeds is proved to be equal to the semi-infinite interval $[s^*,\infty)$, for some $s^*>0$ which is characterized by a variational formula. 
\end{abstract}


\section{Introduction}
\setcounter{equation}{0}

We consider the following diffusive prey-predator model with nonlocal dispersal
\begin{equation}\label{pp}
\begin{cases}
U_t(x,t)= \bN_1[U(\cdot,t)](x)+[aU(1-U)-V](x,t), \; x\in\bR,\, t>0,\\
V_t(x,t)=d \bN_2[V(\cdot,t)](x)+[bV(1-V/U)](x,t), \; x\in\bR,\, t>0,
\end{cases}
\end{equation}
where the unknown functions $U,V$ stand for the population densities of prey and predator species at position $x$ and time $t$, respectively, $d,a,b$ are positive constants such that $1,d$ are diffusion coefficients and $a,b$ are intrinsic growth rates of species $U,V$, respectively. The functional response of predator to prey is normalized to be $1$. The prey obeys the logistic growth and its carrying capacity is normalized to be 1. However, the density of predator follows a logistic dynamics with a varying carrying capacity proportional to the density of prey. Moreover, for $i=1,2$, $\bN_i$ formulates the spatial nonlocal dispersal of individuals and is defined by
\beaa
\bN_i[u(\cdot,t)](x):=\int_{\mathbb{R}}J_i(y)u(x-y,t)dy-u(x,t),\; u=U,V,
\eeaa
where $J_i$ is a probability density function satisfying the following conditions: 
\begin{enumerate}
\item[(H1)]\,  $J_i$ is a nonnegative continuous function defined in $\bR$;
\item[(H2)]\, $\int_{\mathbb{R}}J_i(y)dy=1$ and $J_i(y)=J_i(-y)$ for all $y\in \mathbb{R}$;
\item[(H3)] there exists $\hat{\lambda}_i\in(0,\infty]$ such that  $\int_{\mathbb{R}}J_i(y)e^{\lambda y} dy < \infty$ for any $\lambda \in (0, \hat{\lambda_i})$ and $\int_{\mathbb{R}}J_i(y)e^{\lambda y} dy\to\infty$ as $\lambda \uparrow \hat{\lambda}_i$.
\end{enumerate}
Note that there is always the predator-free state $(1,0)$, and there is the unique constant co-existence state $(a^*,a^*)$, $a^*:=1-1/a\in(0,1)$, when $a>1$.

When the nonlocal dispersal in system \eqref{pp} is replaced by the classical diffusion, the dynamical behaviors of the corresponding system was investigated in the survey paper \cite{CGS21}. In fact, system \eqref{pp} without diffusion (the ODE system) arises in the control of introduced rabbits to protect native birds from introduced cat predation in an island (cf. \cite{CLS}), when we consider the case without rabbits and control. For the detailed biological background of the full ODE system including the rabbits and control, we refer the reader to \cite{CLS,CS}. However, it is reasonable and more realistic to take into account the influence of spatial movements of birds and cats, namely, the effect of diffusion. On the other hand, to model the long range movements and nonadjacent interactions of individuals it is more realistic to consider the nonlocal dispersal instead of the random movement with classical diffusion. This motivates us to study system \eqref{pp}.

The main purpose of this paper is to study the existence and nonexistence of traveling wave solutions to \eqref{pp} connecting the predator-free state and the co-existence state. Here a solution $(U,V)$ to \eqref{pp} is called a traveling wave solution of \eqref{pp}, if there exist a constant $s\in\R$ ({\it the wave speed}) and a function $(\phi,\psi)$ ({\it the wave profile}) of class $C^1(\R)$ such that
$$(U,V)(x,t)=(\phi,\psi)( z),\ \ z:=x-st.$$
We are interested in the traveling waves connecting the predator-free state $(1,0)$ to the co-existence state $(a^*,a^*)$. Therefore, $\{s,\phi,\psi\}$ satisfies the following system of equations:
\begin{equation}\label{TWS}
\left\{
\begin{aligned}
&   \bN_1[\phi](z)+ s\phi'(z) +a\phi(z)\left[1-\phi(z)\right]-\psi(z)    =0, &\; z\in \mathbb{R},\\
&  d \bN_2[\psi](z)+ s\psi'(z) +b\psi(z)\left[1-\frac{\psi(z)}{\phi(z)}\right] =0, & \; z\in \mathbb{R},
\end{aligned}
\right.
\end{equation}
together with the following asymptotic boundary conditions
\begin{equation}\label{BC}
(\phi,\psi)(\infty) = (1,0),\quad (\phi,\psi)(-\infty) = (a^{*},a^{*}).
\end{equation}

For convenience, for $i=1,2$, we set
\beaa
I_i(\lambda):=\int_\bR J_i(y)e^{\lambda y}dy,\ \ \lambda\in\R,
\eeaa
with $I_i(\lambda)\in(0,\infty)$ if $0\le\lambda<\hat{\lambda}_i$ and $I_i(\lambda)=\infty$ if $\lambda\ge\hat{\lambda}_i$. Note that, due to the symmetry of $J_i$, the function $I_i$ is even, and it is also strictly convex in $(-\hat{\lambda}_i,\hat{\lambda}_i)$. Also, we introduce the quantity
\beaa
s^*:=\inf_{\lambda\in(0,\hat{\ld}_2)}\frac{d[I_2(\lambda)-1]+b}{\lambda}.
\eeaa
Note that $s^*$ is well-defined, the infimum is reached, and $s^*>0$.

We now state our main theorem as follows.

\begin{theorem}\label{th:main}
Let $a$, $b$ and $d$ be given positive constants such that $a\ge 4$ and $d<b$. Then there exists a positive solution $(\phi,\psi)$ of \eqref{TWS}-\eqref{BC} for each $s > s^{*}$. This existence also holds for $s=s^*$, if we further assume that $J_2$ has a compact support. Moreover, there exist no positive solutions of \eqref{TWS}-\eqref{BC} for $s < s^{*}$.
\end{theorem}

Hereafter, a function $(\phi,\psi)$ is positive if $\phi,\psi>0$ in $\bR$. A corresponding theorem to Theorem~\ref{th:main}, under the condition $a\ge 4$, on the traveling waves in the classical diffusion case was derived in \cite{CGS21}. Theorem~\ref{th:main} characterizes the minimal speed $s^*$ in the nonlocal dispersal case, but it needs the extra condition $d<b$ in the construction of lower and upper solutions to~\eqref{TWS}.

It is worth to note that there are some difficulties in the study of system \eqref{pp}. It is a non-monotone system which is lack of comparison principle and a singularity may occur in the $V$-equation when $U$ reaches zero in a finite time which has been confirmed numerically in \cite{CLS}. Thus system \eqref{pp} is named as a {\it singular} prey-predator model due to the latter fact. This is in contrast to the system with $U$-equation in \eqref{pp} replaced by
\beaa
U_t(x,t)= \bN_1[U(\cdot,t)](x)+[aU(1-U)-kUV](x,t), \; x\in\bR,\, t>0,
\eeaa
in which the functional response of predation is a linear function of prey. In this case, the strong maximum principle for the scalar equation gives that $U>0$ for all $t>0$. Hence no singularity can occur for the nonlinear term $V/U$ in the $V$-equation.

The existence of wave profiles $(\phi,\psi)$ to \eqref{TWS}-\eqref{BC} for all speeds $s>s^*$ is based on the construction of lower and upper solutions, which is new for this singular nonlocal system. The case $s=s^*$ is more involved, it requires a special care and uses the boundedness of the support of $J_2$. The characterization of the limiting state behind the front, that is, $(\phi,\psi)(-\infty)=(a^*,a^*)$ is carried out with a squeezing method and an argument by contradiction. The proof of the nonexistence of wave profiles for all speeds $s<s^*$ is also done by contradiction. Section~\ref{sec2} is devoted to the existence part in Theorem~\ref{th:main}, while the nonexistence is shown in Section~\ref{sec3}.


\section{Existence of traveling waves}\label{sec2}

This section is devoted to the derivation of the existence of positive traveling waves, under the assumptions $a\ge 4$ and $d<b$. First, we give the definition of upper and lower solutions as follows.

\begin{definition}\label{low-up}
Positive continuous functions $(\overline{\phi },\overline{\psi })$ and $(\underline{\phi},\underline{\psi})$ are called a pair of upper and lower solutions of \eqref{TWS} if $\underline{\phi}( z)\le \overline{\phi}( z) $, $\underline{\psi}( z)\leq\overline{\psi}( z)$ for all $ z \in \mathbb{R}$ and the following inequalities
\begin{eqnarray}
&&\mU_1(z):=\mathcal{N}_1[\overline{\phi }]( z )+s\overline{\phi }^{\prime }( z)+a\ophi(z)[1-\overline{\phi }( z )]-\underline{\psi }( z )\le 0,  \label{u1} \\
&&\mU_2(z):=d \mathcal{N}_2[\overline{\psi }]( z )+s\overline{\psi }^{\prime }( z)+b\overline{\psi }( z )[1-\overline{\psi }( z )/\overline{\phi }( z )]\le 0,  \label{u2} \\
&&\mL_1(z):= \mathcal{N}_1[\underline{\phi }]( z )+s\underline{\phi }^{\prime }( z)+ a\uphi(z)[1-\underline{\phi }( z )]-\overline{\psi }( z )\ge 0,  \label{l1} \\
&&\mL_2(z):= d\mathcal{N}_2[\underline{\psi }]( z )+s\underline{\psi }^{\prime }( z)+ b\underline{\psi }( z )[1-\underline{\psi }( z )/\underline{\phi }( z )]\ge 0  \label{l2}
\end{eqnarray}%
hold for all $z\in\mathbb{R}\setminus E$, where $E$ is some finite subset of $\mathbb{R}$.
\end{definition}

The following result can be proved by applying Schauder's fixed point theorem (cf., e.g., \cite{CGY17,DGLP19,GHW} and references cited therein). We safely omit its proof.

\begin{proposition}\label{le6}
Let $s>0$ be given. Let $(\overline{\phi },\overline{\psi })$ and $(\underline{\phi},\underline{\psi})$ be a pair of bounded upper and lower solutions of \eqref {TWS}. Then system \eqref {TWS} admits a positive solution $(\phi,\psi)$ of class $C^1(\R)$ such that
\begin{equation*}
\underline{\phi }( z )\leq \phi(z)\leq \overline{\phi}( z ),\; \upsi(z)\leq \psi(z)\leq\overline{\psi}( z ),\; z \in \mathbb{R}.
\end{equation*}
\end{proposition}


\subsection{Upper and lower solutions}

To derive the existence of traveling waves, we need to find some suitable pairs of upper and lower solutions of \eqref{TWS}. For this, we divide the analysis into two cases: $s>s^{*}$ and $s=s^{*}$. The main idea of the following construction is motivated by \cite{CGS21,DGLP19}.

\subsubsection{The case $s>s^{*}$}

For a given $s>s^*$, we consider the quantity
\beaa
A(\lambda)=A(\lambda;s):=d[I_2(\lambda)-1]-s\lambda+b,\; \lambda>0.
\eeaa
It follows from the definition of $s^*$ and the strict convexity of $A(\lambda)$ that there are two positive constants $\lambda_1<\lambda_2<\hat{\ld}_2$ such that
$$A(\lambda_1)=A(\ld_2)=0$$
and $A(\lambda)<0$ for all $\lambda\in(\lambda_1,\lambda_2)$. Also, set 
$$B(\ld)=B(\ld;s):=[I_1(\ld)-1]-s\ld.$$ 
Since $B(0)=0$ and $B'(0)=-s<0$, we can choose a constant $\ld_0\in(0,\min\{\ld_1,\hat{\ld}_1\})$ small enough such that
$$B(\ld_0)<0.$$
Now, for a fixed constant $\mu\in(1,\min{\{ \lambda_2/ \lambda_1,2\}})$, we choose 
\be\label{q0}
q>\max\Big\{1,\frac{2b}{-A(\mu\ld_1)}\Big\}.
\ee
Then, the function $f(z) := e^{ -\lambda_1z}-qe^{-\mu \lambda_1z}$ has exactly one zero $z_0>0$ and exactly one maximum point $z_M\in(z_0,\infty)$, and there holds $f(z_M)>0$. Thus, using $a\ge4>1$ and $d<b$, we can choose $\delta$ such that
$$0<\delta<\min\Big\{f(z_M),\underbrace{1-\frac{1}{a}}_{=a^*},\frac12\Big(1-\frac{d}{b}\Big)\Big\},$$
and a point $z_1\in(z_0,z_M)$ such that $f(z_1)=\delta$. Note that
\be\label{delta}
d<b(1-2\delta).
\ee
Furthermore, the inequalities $z_1>z_0>0$, $\ld_1>0$ and $\mu>1$ imply
\be\label{q}
{e^{(\mu-1)\ld_1z_1}>e^{(\mu-1)\ld_1z_0}=q.}
\ee
With these choices of $\mu,q,\delta$, we finally choose $\e$ such that
\be\label{ep1}      
0<\e<\min{\left\{\frac{\delta}{1+s\ld_1+a},\frac{e^{(\mu-1)\lambda_1z_1}-q}{(1+s\lambda_1+a)e^{(\mu-1)\lambda_1z_1}} \right\}}.
\ee
Note that the constant $\e$ is admissible, due to \eqref{q}, and that $0<\e<\delta<1/2$.

Then we introduce the following bounded positive continuous functions
\begin{equation}\label{uplosol}
 \begin{aligned}
  \overline{\phi}(z) &=\left\{
                           \begin{array}{ll}
                             1-\e,                            & \hbox{$z\leq 0$,} \\
                             1-\e e^{ -\lambda_1z}, & \hbox{$z>0$,}
                           \end{array}
                         \right.\quad
  \underline{\phi}(z) =\left\{
                           \begin{array}{ll}
                             \frac{1}{2},                           & \hbox{$z\le 0$,} \\
                             1-\frac{1}{2}e^{ -\lambda_0z}, & \hbox{$z>0$,}
                           \end{array}
                         \right.\\
  \overline{\psi}(z) &=\left\{
                           \begin{array}{ll}
                             1,                            & \hbox{$z\leq 0$,} \\
                             e^{ -\lambda_1z},                  & \hbox{$z>0$,}
                           \end{array}
                         \right.\quad
  \underline{\psi}(z) =\left\{
                           \begin{array}{ll}
                             \delta,                      & \hbox{$z\leq z_1$,} \\
                             e^{ -\lambda_1z}-qe^{-\mu \lambda_1z},      & \hbox{$z>z_1$.}
                           \end{array}
                         \right.
 \end{aligned}
\end{equation}

\begin{lemma}\label{la:upper-lower1}
Assume that $a\ge 4$ and $d<b$. For $s>s^{*}$, the functions $(\overline{\phi},\overline{\psi})$ and $(\underline{\phi},\underline{\psi})$ defined in \eqref{uplosol} are a pair of upper and lower solutions of \eqref{TWS}.
\end{lemma}

\begin{proof}
{\bf Step 1.} Note that $\ophi\le 1$ in $\R$. For $z<0$, $\ophi(z)=1-\e$, whence $\bN[\ophi](z)\le \e$ and so 
\beaa
\mU_1(z)\le [\e+a(1-\e)\e]-\delta<0,\; z<0,
\eeaa
due to $0<\e<\delta/(1+s\ld_1+a)<\delta/(1+a)$. For $z>0$, we have $\ophi(z)=1-\e\elz$ and so $\bN[\ophi](z)\le \e\elz$. Hence
\beaa
\mU_1(z)\le \e e^{-\ld_1z}+\e s\ld_1e^{-\ld_1 z}+a\e e^{-\ld_1 z}-\upsi(z)=\e(1+s\ld_1+a)e^{-\ld_1 z}-\upsi(z),\; z>0.
\eeaa
For $z\in(0,z_1)$, there holds $\upsi(z)=\delta$, which implies that 
$$\mU_1(z)\le \e(1+s\ld_1+a)-\delta<0,$$ 
since $\e<\delta/(1+s\ld_1+a)$. For $z>z_1$, we have $\upsi(z)=\elz-qe^{-\mu\ld_1z}$ and so
\beaa
\mU_1(z)\le e^{-\ld_1 z}[\e (1+s\ld_1+a)-1]+qe^{-\mu\ld_1z}<0,
\eeaa
provided
\be\label{q1}
q<[1-\e(1+s\ld_1+a)]e^{(\mu-1)\ld_1z_1}. 
\ee
But \eqref{q1} holds due to \eqref{ep1}. Therefore, \eqref{u1} holds for all $z\in\bR\setminus\{0,z_1\}$.

{\bf Step 2.} For $z<0$, since $\opsi(z)=1$ and $\opsi\le 1$ in $\bR$, we have $\bN[\opsi](z)\le 0$. Hence 
\beaa
\mU_2(z)\le b(1-1/\ophi(z))\le 0,\; z<0,
\eeaa
due to $0<\ophi\le 1$ in $\R$. For $z>0$, from $\opsi(z)=e^{-\ld_1z}$ and $\opsi(z-y)\le e^{-\ld_1(z-y)}$ for all $y\in\bR$ it follows that
\beaa
\mU_2(z)\le d[I_2(\ld_1)-1]\elz-s\ld_1e^{-\ld_1z}+b\elz=A(\ld_1)\elz=0,
\eeaa
due to $A(\ld_1)=0$. Therefore, \eqref{u2} holds for all $z\in\bR\setminus\{0\}$.

{\bf Step 3.} For $z<0$, we have $\uphi(z)=1/2$. Since $\uphi\ge 1/2$ in $\R$, we obtain $\bN[\uphi](z)\ge 0$ for $z<0$. Then $\mL_1(z)\ge a/4-1\ge 0$ for $z<0$, using $a\ge 4$. For $z>0$, we have $\uphi(z)=1-e^{-\ld_0z}/2\ge1/2$ and $\opsi(z)=\elz$. Since $\uphi(z-y)\ge 1-e^{-\ld_0(z-y)}/2$ for all $y\in\bR$ and $\uphi\ge 1/2$ in $\R$, we get, for $z>0$,
\beaa
\mL_1(z)&\ge& [1- e^{-\ld_0z}  I_1(\ld_0)/2]-(1- e^{-\ld_0z} /2)+s\ld_0 e^{-\ld_0z} /2+(a/2)e^{-\ld_0z} /2-\elz\\
&\ge&\frac{1}{2}e^{-\ld_0z}\{-[I_1(\ld_0)-1-s\ld_0]+a/2-2\}\ge 0,
\eeaa
due to $\elz<e^{-\ld_0z}$, $B(\ld_0)<0$ and $a\ge 4$. Therefore, \eqref{l1} holds for all $z\in\bR\setminus\{0\}$.

{\bf Step 4.} For $z<z_1$, $\upsi(z)=\delta$ and we compute
\beaa
\mL_2(z)\ge -d\delta+b\delta(1-2\delta)\ge 0,
\eeaa
using $\upsi\ge 0$, $\uphi\ge 1/2$ and~\eqref{delta}. For $z>z_1$, since $\upsi(y)\ge e^{-\ld_1y}-qe^{-\mu\ld_1y}$ for all $y\in\bR$, we compute
\beaa
\mL_2(z)&\ge& d\{[\elz I_2(\ld_1)-qe^{-\mu\ld_1z}I_2(\mu\ld_1)]-(\elz-qe^{-\mu\ld_1z})\}\\
&&\quad -s(\ld_1\elz-q\mu\ld_1e^{-\mu\ld_1z})+b(\elz-qe^{-\mu\ld_1z})[1-\upsi(z)/\uphi(z)]\\
&\ge&\elz A(\ld_1)-qe^{-\mu\ld_1z}A(\mu\ld_1)-2be^{-2\ld_1z}\\
&=&\{-e^{-\mu\ld_1z}A(\mu\ld_1)\}\left\{q-\frac{2be^{(\mu-2)\ld_1z}}{[-A(\mu\ld_1)]}\right\}\ge 0,
\eeaa
using $\upsi(z)\le \elz$, $\uphi\ge 1/2$, $A(\ld_1)=0$, $A(\mu\ld_1)<0$, $\mu<2$,
\beaa
e^{(\mu-2)\ld_1z}\le e^{(\mu-2)\ld_1z_1}<1,\; z>z_1>0,
\eeaa
and \eqref{q0}. Here we have used $(\alpha-\beta)^2\le\alpha^2$ for any $\alpha\ge\beta>0$, applied to $\alpha:=e^{-\lambda_1z}$ and $\beta:=e^{-\mu\lambda_1z}$. Therefore, \eqref{l2} holds for all $z\in\bR\setminus\{z_1\}$.

Lastly, it is clear by construction that $1/2\le\uphi\le\ophi<1$ and $0<\upsi\le\opsi\le1$ in $\bR$. The lemma is thereby proved.
\end{proof}


\subsubsection{The case $s=s^*$}

In this case, we have $\ld_1=\ld_2>0$, i.e., $\ld_1$ is a double root of
$$A(\ld;s^*)=0,$$
with necessarily $\lambda_1<\hat{\lambda}_2$. In this case, we also have
\be\label{dA}
d\int_\bR J_2(y)ye^{\ld_1y}dy=s^*.
\ee
Suppose that $J_2$ has a compact support such that its support is contained in $[-S,S]$ for some $S\in(0,\infty)$. We also choose a small positive constant $\ld_0\in(0,\min\{\ld_1,\hat{\ld}_1\})$ such that
$$B(\ld_0;s^*)<0.$$
First, we choose a fixed constant $h>{\lambda_1} e$ large enough such that the equation $hze^{-\ld_1z}=1$ has exactly two roots $z_1,z_2$ with
$$0<z_1<\frac{1}{\ld_1}<z_2<\infty\ \hbox{ and }\ z_2-z_1>S.$$
Note that $hze^{-\ld_1z}>1$ for all $z\in(z_1,z_2)$, and that the map $z\mapsto hze^{-\ld_1z}$ is decreasing in~$[z_2,\infty)$. Secondly, we choose $z_3>z_2$ large enough such that
\be\label{z3}
he^{-\ld_0z_3}\le\frac{a(\ld_1-\ld_0)e}{4}.
\ee
Thirdly, for any positive constant $q$, one can check that the function
$$g(z):=( hz-q\sqrt{z} )e^{ -\lambda_1 z},\quad z\geq 0,$$
has a zero $z_0 := (q/h)^2\in(0,\infty)$ such that $g>0$ in $(z_0,\infty)$, $g(z_0)=0$ and $g$ has a unique maximal point $z_M\in(z_0,\infty)$. Note that $g(z_M)>0$. Furthermore, $g$ is increasing in $[z_0,z_M]$ and decreasing in $[z_M,\infty)$. We choose a constant $q$ large enough such that $z_0>z_2$ and
\be\label{qq}
q>\frac{16bh^2\max_{z>0}\{z^2(z+S)^{3/2}\elz\}}{d\int_\bR J_2(y)y^2e^{\ld_1y}dy}.
\ee
Then we choose $z_4$ in $(z_0,z_M)$ such that $g(z_4) = \delta$, for a fixed $\delta$ satisfying
$$0<\delta<\min\Big\{g(z_M),a^*,\frac12\Big(1-\frac{d}{b}\Big)\Big\},$$
so that~\eqref{delta} still holds. Lastly, with the above choices of $h,q,\delta$, we choose $\e$ such that
\be\label{ep2}
0<\e<\frac{\min\{\delta,hz_4-q\sqrt{z_4}\}}{1+s^*\ld_1+a}.
\ee
Note that $0<\e<\delta<1/2$.

Then we define the bounded positive continuous functions
\begin{equation}\label{uplosol2}
 \begin{aligned}
  \overline{\phi}(z) &=\left\{
                           \begin{array}{ll}
                             1-\e,                            & \hbox{$z\leq 0$,} \\
                             1-\e\elz, & \hbox{$z>0$,}
                           \end{array}
                         \right.\quad
  \underline{\phi}(z) =\left\{
                           \begin{array}{ll}
                               \frac{1}{2},                           & \hbox{$z\leq z_3$,} \\
                             1-\frac{1}{2}e^{-\ld_0(z-z_3)}, & \hbox{$z>z_3$,}
                           \end{array}
                         \right.\\
  \overline{\psi}(z) &=\left\{
                           \begin{array}{ll}
                             1,                             & \hbox{$z\leq z_2$,} \\
                             hze^{ -\lambda_1z},                  & \hbox{$z> z_2$,}
                           \end{array}
                         \right.\quad
  \underline{\psi}(z) =\left\{
                           \begin{array}{ll}
                             \delta,                                      & \hbox{$z\leq z_4$,} \\
                             (hz-q\sqrt{z}) e^{ -\lambda_1z},      & \hbox{$z> z_4$.}
                           \end{array}
                         \right.
 \end{aligned}
\end{equation}

\begin{lemma}\label{la:upper-lower2}
Assume that $a\ge 4$, $d<b$, and that $J_2$ has a compact support such that its support is contained in $[-S,S]$ for some $S\in(0,\infty)$. For $s=s^{*}$, the functions $(\overline{\phi},\overline{\psi})$ and~$(\underline{\phi},\underline{\psi})$ defined in \eqref{uplosol2} are a pair of upper and lower solutions of \eqref{TWS}.
\end{lemma}

\begin{proof}
As before, we also divide our proof into 4 steps.

{\bf Step 1.} Note that $\overline{\phi}\le1$ in $\R$. For $z<0$, we have $\mU_1(z)\le [\e+a(1-\e)\e]-\delta\le 0$, due to \eqref{ep2}. For $z\in(0,z_4)$, we have
\beaa
\mU_1(z)\le \e(1+s^*\ld_1+a)e^{-\ld_1 z}-\delta\le 0,
\eeaa
using $\e<\delta/(1+s^*\ld_1+a)$. For $z>z_4$, we have $\upsi(z)=(hz-q\sqrt{z})\elz$ and so
\beaa
\mU_1(z)\le \{\e (1+s^*\ld_1+a)-(hz-q\sqrt{z})\}\elz\le 0,
\eeaa
due to \eqref{ep2} and $hz-q\sqrt{z}\ge hz_4-q\sqrt{z_4}$ for all $z\geq z_4>z_0$. Therefore, \eqref{u1} holds for all $z\in\bR\setminus\{0,z_4\}$.

{\bf Step 2.} For $z<z_2$, since $\opsi(z)=1$ and $\opsi\le 1$ in $\bR$, we have $\bN[\opsi](z)\le 0$. Hence $\mU_2(z)\le 0$ for $z<z_2$, since $0<\ophi\le 1$ in $\R$. For $z>z_2$, we have $\opsi(z)=hze^{-\ld_1z}$ and 
\beaa
\int_\bR J_2(y)\opsi(z-y)dy=\int_{-S}^S J_2(y)\opsi(z-y)dy\le \int_\bR J_2(y)h(z-y)e^{-\ld_1(z-y)}dy,
\eeaa
since $z-y>z_2-S>z_1$ for all $y\in[-S,S]$ gives $\opsi(z-y)\le h(z-y)e^{-\ld_1(z-y)}$. Hence we get, for $z>z_2$,
\beaa
\mU_2(z)\le \{d[I_2(\ld_1)-1]-s^*\ld_1+b\}(hz\elz)-\left[d\int_\bR J_2(y)ye^{\ld_1y}dy-s^*\right](h\elz)=0,
\eeaa
using $A(\ld_1;s^*)=0$ and \eqref{dA}. Therefore, \eqref{u2} holds for all $z\in\bR\setminus\{z_2\}$.

{\bf Step 3.} For $z<z_3$, we have $\uphi(z)=1/2$. Hence $\mL_1(z)\ge a/4-1\ge 0$ for $z<0$, using $\uphi\ge 1/2$ in $\R$, $\overline{\psi}\le1$ in $\R$, and $a\ge 4$. For $z>z_3$, using $\uphi(z)=1-e^{-\ld_0(z-z_3)}/2$, $\uphi(z-y)\ge 1-e^{-\ld_0(z-y-z_3)}/2$ for all $y\in\bR$, and $\uphi\ge 1/2$ in $\R$, we get
\beaa
\mL_1(z)&\ge& \frac{1}{2}e^{-\ld_0(z-z_3)}\{-[I_1(\ld_0)-1-s^*\ld_0]+a/2\}-hz\elz,\\
&\ge&\frac{1}{2}e^{-\ld_0(z-z_3)}\{a/2-2hz\elz e^{\ld_0(z-z_3)}\}\ge 0,
\eeaa
due to $B(\ld_0;s^*)<0$ and \eqref{z3}. Here we have also used
\beaa
2hz\elz e^{\ld_0(z-z_3)}\le \frac{2he^{-\ld_0z_3}}{(\ld_1-\ld_0)e},\;\forall\, z>0.
\eeaa
Therefore, \eqref{l1} holds for all $z\in\bR\setminus\{z_3\}$.

{\bf Step 4.} For $z<z_4$, $\upsi(z)=\delta$ implies that
\beaa
\mL_2(z)\ge -d\delta+b\delta(1-2\delta)\ge 0,
\eeaa
using $\upsi\ge 0$, $\uphi(z)\ge 1/2$ and \eqref{delta}. Recall that $z_4>z_0>z_2>z_1+S>S$ and $J_2(y)=0$ for $y\in\bR\setminus[-S,S]$. Now, let $z>z_4$. Using $z-y>z_4-y>z_2-S>z_1>0$ for $y\in[-S,S]$ and 
\beaa
\mbox{$\upsi(\xi)\ge (h\xi-q\sqrt{\xi}) e^{ -\lambda_1 \xi}$ for all $\xi\ge 0$,}
\eeaa
we get that
\beaa
&&\int_\bR J_2(y)\upsi(z-y)dy=\int_{-S}^S J_2(y)\upsi(z-y)dy\\
&\ge&\int_{-S}^S J_2(y)[h(z-y)-q\sqrt{z-y}]e^{-\ld_1(z-y)}dy\\
&=&hz\elz\int_{-S}^S J_2(y)e^{\ld_1y}dy-h\elz\int_{-S}^S J_2(y)ye^{\ld_1y}dy-q\elz\int_{-S}^S J_2(y)\sqrt{z-y}e^{\ld_1y}dy\\
&=&hz\elz I_2(\ld_1)-h\elz\int_\bR J_2(y)ye^{\ld_1y}dy-q\elz\int_{-S}^S J_2(y)\sqrt{z-y}e^{\ld_1y}dy.
\eeaa
Then, using $\uphi\ge 1/2$, $A(\ld_1;s^*)=0$ and \eqref{dA}, we obtain
\beaa
\mL_2(z)&\ge& hz\elz\{ d[I_2(\ld_1)-1]-s^*\ld_1+b\}-h\elz\left\{d\int_\bR J_2(y)ye^{\ld_1y}dy-s^*\right\}\\
&&\quad -dq\elz\left\{\int_{-S}^S J_2(y)\sqrt{z-y}e^{\ld_1y}dy - \sqrt{z}\right\}+qs^*\left[\ld_1\sqrt{z}-\frac{1}{2\sqrt{z}}\right]\elz\\
&&\qquad -bq\sqrt{z}\elz-2bh^2z^2 e^{-2\ld_1z}\\
&:=&\elz[qQ_1(z)-Q_2(z)],
\eeaa
where
\beaa
&&Q_1(z):=-d\left\{\int_{-S}^S J_2(y)\sqrt{z-y}e^{\ld_1y}dy - \sqrt{z}\right\}+s^*\left[\ld_1\sqrt{z}-\frac{1}{2\sqrt{z}}\right]-b\sqrt{z},\\
&&Q_2(z):=2bh^2z^2\elz.
\eeaa
Note that, using $A(\ld_1;s^*)=0$ and \eqref{dA}, we get
\beaa
Q_1(z)&=&d\int_{-S}^S J_2(y)[\sqrt{z}-\sqrt{z-y}]e^{\ld_1y}dy-\frac{d}{2\sqrt{z}}\int_{-S}^S J_2(y)ye^{\ld_1y}dy\\
&=&d\int_{-S}^S \left[\sqrt{z}-\sqrt{z-y}-\frac{y}{2\sqrt{z}}\right] J_2(y)e^{\ld_1y}dy\\
&=&d\int_{-S}^S \frac{1}{2\sqrt{z}(\sqrt{z}+\sqrt{z-y})^2}J_2(y)y^2e^{\ld_1y}dy
\ge \frac{d}{8(\sqrt{z+S})^3}\int_\bR J_2(y)y^2e^{\ld_1y}dy.
\eeaa
Hence we conclude that
\beaa
\mL_2(z)\ge \frac{d\elz}{8(\sqrt{z+S})^3}\left(\int_\bR J_2(y)y^2e^{\ld_1y}dy\right)
\left\{q-\frac{16bh^2[z^2(z+S)^{3/2}\elz]}{d\int_\bR J_2(y)y^2e^{\ld_1y}dy}\right\}\ge 0
\eeaa
for $z>z_4$, due to \eqref{qq}. Therefore, \eqref{l2} holds for all $z\in\bR\setminus\{z_4\}$.

One can also check that $1/2\le\uphi\le\ophi<1$ and $0<\upsi\le\opsi\le1$ in $\bR$, and the lemma is thus proved.
\end{proof}


\subsection{Tail limits for $s\ge s^*$}

First, for any $s\ge s^*$, the existence of a $C^1(\R)$ solution $(\phi,\psi)$ to~\eqref{TWS} such that $1/2\le\underline{\phi}\le\phi\le\overline{\phi}<1$ and $0<\underline{\psi}\le\psi\le\overline{\psi}\le1$ in $\R$ follows from Lemmas~\ref{la:upper-lower1}-\ref{la:upper-lower2} and Proposition~\ref{le6}. Moreover, by the construction of upper and lower solutions, it is clear that $(\phi,\psi)(\infty)=(1,0)$.

Next, we claim that the limit $(\phi,\psi)(-\infty)$ exists, by applying the method used in \cite{w21}. For the reader's convenience, we provide an outline as follows. To show the limit exists, we let
\beaa
\phi^-:=\liminf_{z\to -\infty}\phi(z),\ \ \phi^+:=\limsup_{z\to -\infty}\phi(z),\ \ 
\psi^-:=\liminf_{z\to -\infty}\psi(z),\ \ \psi^+:=\limsup_{z\to -\infty}\psi(z).
\eeaa
By the construction of upper-lower solutions, we have 
$$1/2\le\phi^-\le\phi^+\le 1\ \hbox{ and }\ 0<\delta\le\psi^-\le\psi^+\le1,$$
for some $\delta\in(0,a^*)$, in both cases $s>s^*$ and $s=s^*$.

Now, from~\eqref{TWS} and the positivity of $s$, both functions $\phi$ and $\psi$ have bounded derivatives $\phi'$ and $\psi'$, and by differentiating~\eqref{TWS}, the functions $\phi$ and $\psi$ are of class $C^2(\R)$ with bounded second-order derivatives ($\phi$ and $\psi$ are actually of class $C^\infty(\R)$). From the definition of $\psi^+$, there exists then a sequence $(z_n)_{n\in\N}$ diverging to $-\infty$ such that $\psi(z_n)\to\psi^+$, $\psi'(z_n)\to0$, and $\phi(z_n)\to\gamma$ as $n\to\infty$, for some $\gamma\in[\phi^-,\phi^+]\subset[1/2,1]$. By evaluating the second equation in~\eqref{TWS} at $z_n$ and passing to the limit as $n\to\infty$, one gets, by the same argument as that for \cite[(2.6)]{w21}, that $b\psi^+(1-\psi^+/\gamma)\ge0$, whence
$$\psi^+\le\phi^+.$$ 
We omit the details here. Similarly, one has
$$\phi^-\le\psi^-.$$
On the other hand, by considering a sequence $(\zeta_n)_{n\in\N}$ diverging to $-\infty$ such that $\phi(\zeta_n)\to\phi^-$, $\phi'(\zeta_n)\to0$, and $\psi(\zeta_n)\to\Gamma$ as $n\to\infty$, for some $\Gamma\in[\psi^-,\psi^+]$, by evaluating the first equation in~\eqref{TWS} at $\zeta_n$ and passing to the limit as $n\to\infty$, one gets that $a\phi^-(1-\phi^-)-\Gamma\le0$, whence
\be\label{phi-}
a\phi^-(1-\phi^-)\le\psi^+.
\ee
Similarly, one has
\be\label{phi+}
a\phi^+(1-\phi^+)\ge\psi^-.
\ee
Since $\psi^->0$ and $\phi^+\le1$, one infers that $\phi^+<1$, whence $\psi^+\le\phi^+<1$. Therefore, $a\phi^-(1-\phi^-)\le\psi^+<1$ and, since $a\ge4$ and $\phi^-\ge1/2$, one gets that $\phi^->1/2$. Finally,
\be\label{t1}
1/2<\phi^-\le\psi^-\le\psi^+\le\phi^+<1.
\ee

Let now
$$D:=\big\{\nu\in[0,1]\mid l(\nu)<\phi^-\le\phi^+<r(\nu)\big\},$$
where
\beaa
l(\nu):=\nu a^*+(1-\nu)/2,\ \ r(\nu):=\nu a^*+1-\nu,\ \ \nu\in[0,1].
\eeaa
Note that $1/2<a^*<1$. Hence $l(\nu)$ is increasing and $r(\nu)$ is decreasing in $\nu\in[0,1]$, and $l(1)=r(1)=a^*=1-1/a$. Secondly, due to \eqref{t1}, $0\in D$ and so, by continuity of~$l(\nu)$ and~$r(\nu)$ with respect to $\nu$, $\nu_0:=\sup D$ is well-defined such that $\nu_0\in(0,1]$.
For contradiction, we assume that $\nu_0<1$. 
Then, by passing to the limit,
\be\label{phi+-}
l(\nu_0)\le\phi^-\le\phi^+\le r(\nu_0),
\ee
and either $\phi^-=l(\nu_0)$ or $\phi^+=r(\nu_0)$. Then we divide the discussion into two cases.

{\it Case 1: $\phi^-=l(\nu_0)$}. From~\eqref{phi-} and~\eqref{t1}-\eqref{phi+-}, one knows that
$$al(\nu_0)[1-l(\nu_0)]=a\phi^-(1-\phi^-)\le\psi^+\le\phi^+\le r(\nu_0),$$
whence
$$\omega:=a l(\nu_0)[1-l(\nu_0)]-r(\nu_0)\le0.$$
But a straightforward calculation yields
$$\omega=(1-\nu_0)\,\Big[\frac{a}{4}-1+\frac{\nu_0(a-2)^2}{4a}\Big],$$
whence $\omega>0$, since $0<\nu_0<1$ and $a\ge4$. Therefore, case~1 is ruled out.

{\it Case 2: $\phi^+=r(\nu_0)$}. Similarly, from~\eqref{phi+}-\eqref{phi+-}, one knows that
$$a r(\nu_0)[1-r(\nu_0)]=a\phi^+(1-\phi^+)\ge\psi^-\ge\phi^-\ge l(\nu_0),$$
whence
$$\theta:=a r(\nu_0)[1-r(\nu_0)]-l(\nu_0)\ge0.$$
But a straightforward calculation yields
$$\theta=-(1-\nu_0)\,\Big[\frac{1}{2}-\frac{\nu_0}{a}\Big],$$
whence $\theta<0$, since $0<\nu_0<1$ and $a\ge4>2$. Therefore, case~2 is ruled out too.

These contradictions entail that $\nu_0=1$ and so the limit $\phi(-\infty)$ exists such that $\phi(-\infty)=a^*$. Finally, it follows from \eqref{t1} that $\psi(-\infty)=a^*$. This proves that $(\phi,\psi)(-\infty)=(a^*,a^*)$ and the proof of the existence part of Theorem~\ref{th:main} is thus complete.\qed


\section{Nonexistence}\label{sec3}

In this section, we provide a proof of the nonexistence part of Theorem~\ref{th:main} as follows. Although it is almost the same as that in \cite{DGLP19}, for the reader's convenience we give the details here.

\begin{theorem}
If $s<s^*$, then \eqref{TWS}-\eqref{BC} does not admit any positive solution $(\phi,\psi)$.
\end{theorem}

\begin{proof}
We argue by contradiction by assuming that \eqref{TWS} admits a positive solution $(\phi,\psi)$ satisfying \eqref{BC} and for some wave speed $s_0<s^*$. From the continuity and positivity of $\phi$ together with~\eqref{BC}, there is $\e>0$ such that $\phi\ge\e>0$ in $\R$. Then the $\psi$-equation of \eqref{TWS} gives that
\beaa
-s_0\psi'(z)\ge d\bN_2[\psi](z)+b\psi(z)[1-\psi(z)/\e],\; z\in\bR.
\eeaa
Hence the $C^1(\R\times\R)$ function $V(x,t):=\psi(x-s_0t)$ satisfies $L[V](x,t)\ge 0$, where
\begin{equation}\label{au-2}
L[V](x,t):=V_t(x,t)-\big\{d\bN_2[V(\cdot,t)](x)+bV(x,t)[1-V(x,t)/\e]\big\},\;x\in\bR,\ t>0,
\end{equation}
and $V(x,0)=\psi (x)\not\equiv 0$ for $x\in\bR$.

Next, since $\hat{s}:=(s_0+s^*)/2<s^*$ and $V(\cdot,0)\ge0$ with $\liminf_{x\to-\infty}V(x,0)>0$, it follows from the comparison principle and the spreading dynamics for the nonlocal logistic equation $L[v]=0$ (cf. \cite{jz}) that
\beaa
\liminf_{ t\to \infty}\psi (\hat{s}t-s_0t)=\liminf_{t\to \infty}V\left(\hat{s}t,t\right)\ge\e>0.
\eeaa
However, $z:=\hat{s}t-s_0t=(s^*-s_0)t/2\to \infty$ as $t\to\infty$, a contradiction to $\psi(\infty)=0$. Thus the theorem is proved.
\end{proof}


\end{document}